\documentclass[11pt]{amsart}
\usepackage{amssymb,amsmath,amsthm,latexsym}
\usepackage[all]{xy}
\usepackage{amssymb,graphicx}
\newtheorem{theorem}{Theorem}

\newtheorem{definition}[theorem]{Definition}

\newtheorem{lemma}[theorem]{Lemma}
\newtheorem*{2D}{2D Conjecture}

\newtheorem{proposition}[theorem]{Proposition}

\begin{document}

\title{2D problems in groups} 

\author{Aditi Kar \and Nikolay Nikolov}
\begin{abstract}
We investigate a conjecture about stabilisation of deficiency in finite index subgroups and relate it to the D2 Problem of C.T.C. Wall and the Relation Gap problem. We verify the pro-$p$ version of the conjecture, as well as its higher dimensional abstract analogues.
\end{abstract}
\maketitle

Given a finitely presented group $G$, the deficiency $\delta(G)$ of $G$ is defined as the maximum of $|X|-|R|$ over all presentations $G=\langle X \ | \ R \rangle$. We related deficiency of a group with 2-dimensionality in \cite{KN} and proposed the following conjecture. 

\begin{2D}[\cite{KN}] \label{stabilisation} Let $G$ be a residually finite finitely presented group such that $\delta (H)-1= [G:H]( \delta(G)-1)$ for every subgroup $H$ of finite index in $G$. Then $G$ has a finite $2$-dimensional classifying space $K(G,1)$.
\end{2D} 

In this paper, we relate the above conjecture with two well-known problems in topological group theory: \emph{Wall's D2 problem} and the \emph{Relation Gap problem}. The main purpose of the paper is to explain the implications  

\[ \textrm{affirmative D2 problem} \Rightarrow \textrm{no relation gap} \Rightarrow \textrm{2D conjecture} \]  

\section{Background}
Let $G$ be a finitely presented group. Set $d(G)$ to be the cardinality of a minimal generating set of $G$.

We denote by $b_i(G)=\dim_\mathbb Q H_i(G,\mathbb Q)$ and note that $\delta(G) \leq b_1(G) \leq d(G)$. 
Starting with a presentation $\langle X|R\rangle$ for $G$, one obtains a Schreier presentation for $H$ with $[G:H](|X|-1)+1$ generators and 
$[G:H]|R|$ relations showing that
\[ \delta(H)-1 \geq [G:H](\delta(G)-1). \]
We are interested in the situation when the above inequality is in fact equality for every finite index subgroup $H$ of $G$.

  We next introduce the invariant $\mu_n(G)$ of Swan \cite{swan}. 
  Let $n \in \mathbb N$. A partial free resolution of $\mathbb Z$ of length $n$ is an exact sequence  
 
 \begin{equation} \label{freeres} \mathcal{F}:\quad  (\mathbb Z G)^{f_n} \rightarrow (\mathbb Z G)^{f_{n-1}} \rightarrow \cdots \rightarrow (\mathbb Z G)^{f_0} \rightarrow \mathbb{Z}\rightarrow 0\end{equation}
 
 and we define $\mu_n(\mathcal{F})=\sum_{i=0}^{n} (-1)^{n-i}f_i$.
 
 Recall the well-known Morse inequalities.
 \begin{proposition} \label{morse}
 	Let $n \in \mathbb N$ and $\mathcal F$ be a partial free resolution (\ref{freeres}) as above. Then \[\sum_{i=0}^n (-1)^{n-i} b_i(G) \leq \mu_n(\mathcal F).\]
 \end{proposition} 
 R. Swan \cite{swan} defined the following invariant while studying free resolutions of modules of finite groups.
 \begin{definition} Let $n \in \mathbb N$.
 The invariant $\mu_n(G)$ is defined as the minimum of $\mu_n(\mathcal F)$ as $\mathcal F$ ranges over all partial free resolutions $\mathcal{F}$ of $\mathbb Z$. 
 \end{definition}

 Given a presentation of $G$ with $e_1$ generators and $e_2$ relations one has the partial free resolution 
 \begin{equation} \label{chain} (\mathbb Z G)^{e_2} \stackrel{\partial_2}{\longrightarrow} (\mathbb Z G)^{e_1} \stackrel{\partial_1}{\longrightarrow} \mathbb Z G \stackrel{\partial_0}{\longrightarrow} \mathbb Z \rightarrow 0
 \end{equation}
 arising as the cellular chain complex of the universal cover of the presentation complex of $G$.
 By taking a presentation which realizes the deficiency of $G$ we obtain $\mu_2(G) \leq 1-\delta(G)$. 
  The case $n=2$ of the Morse inequalities applied to (\ref{chain}), together with $b_0(G)=1$ gives the well-known inequality $\delta(G) \leq b_1(G)-b_2(G)$.

\subsection{Groups with two dimensional classifying spaces.}

The deficiency is easy to compute for groups which have finite two-dimensional classifying spaces. Examples of such groups are surface groups or more generally, torsion-free one relator groups and direct products of two free groups. 

\begin{lemma}\label{space} If a group $G$ has a finite two-dimensional space $K(G,1)$, then $\delta(G)=1 -\chi(G)$ and consequently, $\delta (H)-1= [G:H] (\delta (G)-1)$ for every subgroup $H$ of finite index in $G$.
\end{lemma} 

For example $\delta(F_n \times F_m) = -(n-1)(m-1)$ while the deficiency of a torsion-free one relator group defined on $d$ generators is $d-2$. 

The 2D Conjecture stated in the introduction proposes that the converse of Lemma \ref{space} holds. Note that its 1-dimensional analogue is true as shown by R. Strebel \cite{strebel} (see also \cite[Theorem 7]{AJN} for a different perspective).

\begin{proposition}[\cite{strebel}] \label{1dim} Let $G$ be a finitely generated residually finite group. Then $G$ is a free group if and only if $d(H)-1= |G:H|(d(G)-1)$ for every subgroup $H$ of finite index in $G$.
\end{proposition} 

Strebel proved Proposition \ref{1dim} as an answer to a question of Lubotzky and van den Dries \cite{LD}, who had shown that its analogue does not hold in the class of profinite groups. At the same time Lubotzky \cite[Proposition 4.2]{L} proved that the analogue of Proposition \ref{1dim} is true in the class of pro-$p$ groups. We will return to pro-$p$ groups in section \ref{p} below.


We remark that the 2D conjecture is closely connected with gradients in groups and their $L^2$ cohomology. The following basic result characterizes groups $G$ with two dimensional classifying spaces in terms of their $L^2$ Betti numbers $\beta_i(G)$.

\begin{lemma}[\cite{KN}] \label{criterion} Let $G$ be an infinite finitely presented group. Then $\delta (G)-1 \leq \beta_1(G) - \beta_2(G)$ with equality if and only if $G$ has a two dimensional classifying space.
\end{lemma}

In particular any counterexample to the 2D conjecture must be a group $G$ with \emph{deficiency gradient} strictly less than $\beta_1(G) - \beta_2(G)$, see \cite{KN} for more details on this connection.

\section{Wall's D2 Problem}

Wall's D2 problem is a generalisation of the Eilenberg Ganea Conjecture and belongs to the class of questions that explore links between homological and geometric dimensions. A finite CW-complex $X$ is said to be a D2 complex if it has cohomological dimension 2. 
The D2 Problem for a finitely presented group $G$ asks if every finite D2 complex with fundamental group $G$ is homotopy equivalent to a finite 2-complex. If the answer is affirmative we shall say that $G$ has the \emph{D2 property}. The problem was proposed by C.T.C. Wall in 1965 \cite{wall} and little is known about it except in the case when $G$ is finite, free or abelian, see \cite{johnson}.

The Eilenberg-Ganea Conjecture asks if every group of cohomological dimension 2 is of geometric dimension 2. Note that a group of cohomological dimension 2 does not necessarily have a finite classifying space, as famously shown by M. Bestvina and N. Brady \cite{BB}. However, if one assumes that a group $G$ of cohomological dimension 2 has a finite classifying space $X$, then $X$ is a D2 complex. If in addition $G$ has the D2 property, then $X$ is homotopy equivalent to a finite 2-complex. So, $G$ has geometric dimension two, as predicted by Eilenberg-Ganea.  

\section{The Relation Gap problem} 

Suppose that a finitely presented group $G$ is given by the quotient $F/N$ where $F$ is free on the group generators $X$ and $N$ is normally generated in $F$ by the relators $R \subset F$. The action of $F$ by conjugation on $N$ induces an action of $G$ on the abelianisation $N^{ab}$ of $N$. This makes $N^{ab}$ into a $G$-module called the relation module of the presentation. Evidently, the $G$-module $N^{ab}$ can be generated by $|R|$ elements and so the $G$-rank of $N^{ab}$, written $d_G(N^{ab})$, satisfies $d_G(N^{ab}) \leq d_F(N)$, where $d_F(N)$ is the minimum number of normal generators required for $N$. 

A presentation is said to have a relation gap if  $d_G(N^{ab}) \neq d_F(N)$ and the relation gap problem asks, if there exists a finitely presented group with a relation gap. As with the D2 problem, very little is known about the relation gap problem and most proposed counterexamples are not torsion-free, see \cite{Har}.

We give a proof to the following. 

\begin{theorem} \label{D2relgap}A finitely presented group $G$ with the D2 property does not have a relation gap for presentations realizing $\delta(G)$. 
\end{theorem}

This may be known to topological group theorists but we have not found it in the literature. There is a result of Dyer \cite[Theorem 3.5 ]{dyer} with the same statement but with the additional hypothesis $H^3(G, \mathbb Z G)=0$. 

We need the following.
\begin{proposition}[\cite{ji} Proposition 4.3, or \cite{IH}, Remark 1.3] \label{mu2} Let $G$ be a finitely presented group with the D2 property. Then $\mu_2(G)=1-\delta(G)$.
\end{proposition}	

For completeness we give a proof of Proposition \ref{mu2} following \cite{IH}, based on the following theorem of Wall.  

\begin{theorem}[\cite{wall}, Theorem 4] \label{CTCWall}
	Let $X$ be a connected CW-complex, $G= \pi_1(X)$ and let $A_*$ be a positive free chain complex equivalent to the cellular chain complex $C^c_*(X)$ of the universal cover of $X$. Let $K^2$ be a connected CW-complex with fundamental group $G$. There exists another CW complex $Y$ and a homotopy equivalence $h: Y \rightarrow X$ such that $Y$ is obtained from $K^2$ by adding $2$-cells and $3$-cells at the base point to obtain a D2 complex $Y_0$ and then further cells such that $C_*^c(Y,Y_0)$ is the part of $A_*$ in dimension $\geq 3$.
	
	If the symbol $\alpha_i$ denotes the number of $i$-cells or of generators in degree $i$ then
	
	\[\alpha_2(Y_0-K^2)=\alpha_2(A)+ \alpha_1(K)+ \alpha_0(A),\]
	\[ \alpha_3(Y^0-K^2)=\alpha_2(K)+ \alpha_1(A)+ \alpha_0(K).\]
\end{theorem}

\begin{proof}[Proof of Proposition \ref{mu2}]
Let \[ (\mathbb Z G)^{f_2} \rightarrow (\mathbb Z G)^{f_1} \rightarrow (\mathbb Z G)^{f_0} \rightarrow \mathbb Z \rightarrow 0\]
be a partial free resolution of $\mathbb Z$ with $f_2-f_1+f_0=\mu_2(G)$. Extend this to a free resolution $A_*$ and let $X$ be a CW complex which is a classifying space for $G$. Now $A_*$ is homotopy equivalent to the cellular complex $C^c_*(X)$ of $\tilde X$ and therefore starting with any finite presentation complex $K^2$ for $G$ we can apply Theorem \ref{CTCWall} above. In particular there exists a finite 3-dimensional $D2$ complex $Y_0$ with $\pi_1(Y_0)=G$ and we compute
 \[ \chi(Y_0)=\sum_{i=0}^3 (-1)^{i}\alpha_i(Y_0)= \sum_{i=0}^2 (-1)^i\alpha_i(A_*)=\mu_2(G). \] 
 
We are assuming that the $D2$ Problem has positive solution for $G$, therefore $Y_0$ is homotopy equivalent to a finite $2$-dimensional complex $L$. We have $G=\pi_1(K)=\pi_1(L)$ and $\chi(L)=\chi(Y_0)= \mu_2(G)$. Hence \[ \delta(G) -1 \geq \alpha_1(L)-\alpha_0(L)-\alpha_2(L)= -\chi (L)=-\mu_2(G). \] Therefore $1-\delta(G) \leq \mu_2(G)$. Since the opposite inequality $\mu_2(G) \leq 1-\delta(G)$ always holds we have equality.
\end{proof} 

\begin{proof}[Proof of Theorem \ref{D2relgap}]
Let $G$ be a group with the D2 property. Take a presentation $\langle X \ | \ R \rangle$ for $G$ with $e_1$ generators and $e_2$ relations such that $e_1-e_2= \delta(G)$. We have $G \cong F/N$ where $F$ is a free group of rank $e_1$ on $X$ and $N$ is the normal closure of the relations $R$. Since $e_1-e_2$ realises the deficiency of $G$ it follows that $e_2=d_F(N)$. Let $M=N^{ab}$ be the relation module of this presentation.  Recall the chain complex (\ref{chain}) above. We have $M \cong \ker \partial _1= \mathrm{im} \partial_2$. If $M$ has relation gap then $u:=d_G(M)<e_2$ and in particular there is a surjection of $\mathbb ZG$ modules
	$f:(\mathbb Z G)^u \rightarrow \ker \partial_1$.
	Therefore we can amend the partial resolution above to
	\[ (\mathbb Z G)^{u} \stackrel{f}{\longrightarrow} (\mathbb Z G)^{e_1} \stackrel{\partial_1}{\longrightarrow} \mathbb Z G \stackrel{\partial_0}{\longrightarrow} \mathbb Z \rightarrow 0.
	\]
	
	This gives $\mu_2(G) \leq 1+u-e_2 < 1- \delta(G)$ contradicting Proposition \ref{mu2}. Therefore presentations of $G$ which realize $\delta(G)$ have no relation gap.

\end{proof}

\section{Relation Gap problem v.s. 2D Conjecture}

\begin{theorem} \label{rel2D}
If $G$ is a counterexample to the 2D conjecture then there exists a finite index subgroup $H$ of $G$ such that $H$ has a presentation with relation gap. 
\end{theorem}

\begin{proof}
Suppose that $G$ is a finitely presented group; assume that $X$ is a presentation 2-complex for $G$ realising the deficiency $\delta(G)$. If $X$ is not aspherical, then by Whitehead's Theorem, $H_2(\tilde{X}) \neq 0$. Let $e_i$ denote the number of $i$-cells in $X$. So $\delta(G)-1= e_1-e_2-1$. We have the exact sequence of $G$-modules 

\[\mathcal{F}:\ 0\longrightarrow H_2(\tilde{X})  \longrightarrow \mathbb{Z}G^{e_2} \stackrel{\partial_2}{\longrightarrow} \mathbb{Z}G^{e_1} \stackrel{\partial_1}{\longrightarrow} \mathbb{Z}G \longrightarrow \mathbb{Z} \longrightarrow 0\]

\noindent where $H_2(\tilde X)= \ker \partial_2$. The relation module $R$ associated to $X$ is isomorphic to $\ker \partial_1= \mathrm{im} \partial_2 \cong \mathbb{Z}G^{e_2}/H_2(\tilde{X})$. Take a non-zero element $\rho$ of $H_2(\tilde{X})$. As an element of $\mathbb{Z}G^{e_2}$, $\rho$ has a representation as a non-zero tuple $(a_1,\ldots, a_{e_2})$, where each $a_i$ is a linear combination in $\mathbb{Z}G$ with support $C_i$ as follows: 
\[a_i=\sum_{g \in C_i}a^i_g g\]

Let $C=\cup_i C_i$; this is a finite collection of elements of $G$. There exists a finite index normal subgroup of $G$, say $H$ such that the elements of $C$ project to distinct cosets in $G/H$. The natural structure of $\mathbb{Z}G$ as a $\mathbb{Z}H$-module makes $\mathcal{F}$ into the chain complex for the action of $H$ on $\tilde{X}$. Let $E$ be a collection of coset representatives for $H$ in $G$ such that $C\subseteq E$. Consider 
\[\mathbb{Z}G^{e_2} = \left( \bigoplus_{g \in E} \mathbb{Z}H.g \right)^{e_2} \cong \mathbb{Z}H^{e_2[G:H]}\] 

Let $d$ be the greatest common divisor of the integers $\{a^i_g\ |\ g \in C_i,\ i=1,2,\ldots, e_2\}$. Then $\rho=d\rho'$, where $\rho' \in \mathbb{Z}G^{e_2}$ and all its coefficients are co-prime. As $\rho$ is an element of $\ker \partial_2$ and $\partial_2$ is a homomorphism of torsion-free abelian groups, we deduce that $\rho'$ is also an element of $\ker \partial_2$. Therefore, we can assume that $d=1$. 

Consider the presentation for $H$ arising from the action of $H$ on $\tilde{X}$: this presentation has $(e_1-1)[G:H]+1$ generators and $e_2[G:H]$ relations.
The relation module $R'$ for this presentation of $H$ is the restriction $R\downarrow ^G _H$ of the relation module $R$, wherein $\rho$ represents the zero element. We have assumed that the coefficients of $\rho$ are co-prime and so $\rho$ is a primitive element in the abelian group $(\mathbb{Z} E)^{e_2}$ containing its support in $\mathbb{Z}G^{e_2} \cong \mathbb{Z}H^{e_2[G:H]}$. Consequently $R'\cong\mathbb{Z}H^{e_2[G:H]}/H_2(\tilde{X})$ can be generated by fewer than $e_2[G:H]$ elements as an $H$-module. 
 If the above presentation of $H$ has no relation gap then it needs strictly fewer than $e_2[G:H]$ relations and hence $\delta(H)-1 > [G:H](e_1-e_2-1)= [G:H](\delta(G)-1)$, contradiction.
 
 Therefore if $X$ is not aspherical some finite index subgroup of $G$ has a relation gap.
\end{proof} 

We note that the argument above gives the following general criterion for freeness of $\mathbb ZG$-modules.

\begin{proposition}\label{free}
	Let $G$ be a residually finite group and let $M$ be a finitely generated $\mathbb Z G$-module. Assume that $M$ is torsion free as an abelian group and let $f: (\mathbb ZG)^r \rightarrow M$ be a surjective homomorphism of $\mathbb Z G$ modules. Then $f$ is an isomorphism if and only of $d_H(M)=r [G:H]$ for each subgroup $H$ of finite index in $G$.
	
	In particular $M$ is a free module if and only if $d_H(M)=[G:H]d_G(M)$ for each subgroup $H$ of finite index in $G$.
\end{proposition}
\begin{proof}
If $f$ is not injective we can find an element $\rho=(a_1, \ldots, a_r) \in \ker f$ with support $C= \cup_{i=1}^r C_i$ and coefficients $a^i_g \in \mathbb Z$ defined by
$a_i = \sum_{g \in C_i} a^i_g g$. Since $M$ is torsion free we can assume that the greatest common divisor of all integers $a^i_g$ is 1. There is a finite index subgroup $H$ of $G$ such that $C$ projects injectively into $G/H$ and arguing in the same way as in the proof of Theorem \ref{rel2D} we deduce $d_H(M)< r[G:H]$, contradiction. Therefore $f$ is a bijection and $M$ is a free module.   
\end{proof}

\section{The 2D conjecture for pro-$p$ groups.} \label{p}

In this section $G$ denotes a finitely presented pro-$p$ group, where we consider presentations in the category of pro-$p$ groups. We keep the notation $\delta(G)$ for the maximum of $|X|-|R|$ over all pro-$p$ presentations $\langle X , R\rangle$ of $G$.

Below we prove the analogue of the $2D$ conjecture for $G$:

\begin{theorem} \label{pro-p}
Let $G$ be a finitely presented pro-$p$ group. The following are equivalent:

(i) $\delta(G)-1= [G:H](\delta(H)-1)$ for every open subgroup $H$ of $G$.

(ii) $cd_p(G) \leq 2$.
\end{theorem}  

It will be interesting to find a characterizetion of the finitely presented profinite groups $G$ for which  the condition \emph{(i)} above holds. Note that already the 1-dimensional situation for profinite groups is quite different. See \cite{LD} for examples of profinite groups which satisfy Schreier's rank-index formula for all open subgroups, but are not projective.

\begin{proof}
	For pro-$p$ groups $\delta(G)= \dim_{\mathbb F_p} H^1(G)- \dim_{\mathbb F_p} H^2(G)$ where we write $H^i(G)=H^i(G,\mathbb F_p)$, see \cite[ I.4.2 \& I.4.3]{serre}. Hence, if $cd_p(G) \leq 2$ then $\delta(G)-1= -\chi(G)$, the pro-$p$ Euler characteristic of $G$ and therefore (1) holds.
	
	Conversely, suppose that (1) holds and let $e_i= \dim_{\mathbb F_p} H^i(G)$ for $i=1,2$. We have the partial free resolution
		\[ \mathbb F_p[[G]]^{e_2} \stackrel{d_2}{\longrightarrow} \mathbb F_p[[G]]^{e_1} \stackrel{d_1}{\longrightarrow} \mathbb F_p[[G]] \longrightarrow \mathbb F_p \longrightarrow 0, \]
	arising from the presentation of $G$ with $e_1$ generators and $e_2$ relations. We claim that $J:= \ker d_2$ must be zero. Suppose not. Then we can find an open normal subgroup $N$ of $G$ such that the image $\bar J$ of $J$ under the reduction $(\mathbb F_p[[G]])^{e_2} \rightarrow (\mathbb F_p[G/N])^{e_2}$ is non-zero.
	
	Note that the free $\mathbb F_p[[G]]$ resolution above is also a partial free resolution of $\mathbb F_p[[N]]$ modules.
	We apply the functor $\mathrm{Hom}_{N}(-, \mathbb F_p)$ to the above resolution, using $\mathrm{Hom}_N (\mathbb F_p G, \mathbb F_p)\simeq (\mathbb F_p[G/N])^*$, where by $V^*$ we denote the dual of the vector space $V$ over $\mathbb F_p$. We obtain the chain complex
	
	\[   0 \leftarrow \bar{J}^* \stackrel{d'_3}{\longleftarrow} (\mathbb F_p[G/N]^*)^{e_2} \stackrel{d'_2}{\longleftarrow} (\mathbb F_p[G/N]^*)^{e_1} \stackrel{d'_1}{\longleftarrow} \mathbb F_p[G/N]^* \leftarrow 0.   \]
	which is exact at $\bar J^*$ and whose homology group in degree $i$ is $H^i(N)$ 
	Therefore \[ \delta(N)-1=\sum_{i=0}^2 (-1)^{i+1} \dim H^i(N) =\] \[= (e_1-e_2-1)[G:N] + \dim \bar J^* > [G:N](\delta(G)-1),\]
	since $\bar J^* \not =\{0\}$, a contradiction to (i). Therefore $J=\{0\}$ and $cd_p(G)\leq 2$.
\end{proof}

\section{Higher dimensional analogues}

Deficiency can be viewed as one of the partial Euler characteristics, which are defined as follows:

Let $n \geq 2$ be an integer and let $G$ be a group of type $F_n$. Define $\nu_n(G)$ to be the minimum of $(-1)^n\chi(X)$ where $X$ is a finite CW complex of dimension $n$ such that $\pi_1(X)=G$ and $\pi_i(X)=\{0\}$ for $i=2,3, \ldots, n-1$ (i.e its universal cover $\tilde X$ is $(n-1)$-connected. 
Note that $\nu_2(G)= 1-\delta(G)$ and for completeness we define $\nu_1(G)=d(G)-1$. From the definition of $\nu_n$ and $\mu_n$ we have $\nu_n(G) \geq \mu_n(G)$ for all $n$. We note that Theorem \ref{CTCWall} above implies 

\begin{proposition}\label{equal} $\nu_n(G)=\mu_n(G)$ when $n \geq 3$.
\end{proposition}

Here we prove the higher dimensional analogue of the 2D conjecture.
\begin{theorem}\label{>2}
	Let $n>2$ be an integer and let $G$ be a residually finite group of type $F_n$. Then $G$ has finite classifying space of dimension $n$ if and only if 
	$\nu_n(H)=\nu_n(G)[G:H]$ for every subgroup $H$ of finite index in $G$. 
\end{theorem}
\begin{proof}
	Suppose that $X$ is an $n$-dimensional $K(G,1)$ complex for $G$, then $\nu_n(G) \leq (-1)^n\chi(X)$ from the definition of $\nu_n(G)$. On the other hand the Morse inequalities give $\nu_n(G) \geq \sum_{i=0}^n (-1)^{n-i}b_i(G)= (-1)^n \chi(X)$.
	Therefore $\nu_n(G)=(-1)^n \chi(X)$ and in the same way $\nu_n(H)=(-1)^n \chi(X')$, where $X'$ is the cover of $X$ corresponding to $H$. Since $\chi(X')=[G:H] \chi(X)$ the equality follows.
	
	For the other direction we could use Proposition \ref{equal}. Instead we take a more elementary approach and argue directly using Proposition \ref{free}. 
	
	Suppose that $\nu_n(H)=\nu_n(G)[G:H]$ for every subgroup $H$ of finite index in $G$.
	Let $X$ be the $n$-dimensional CW complex which realises $\nu_n(G)$. Let $e_i$ be the number of $i$-dimensional cells of $X$ and let 
	\[ \quad F_n \stackrel{\partial_n}{\longrightarrow} F_{n-1} \stackrel{\partial_{n-1}}{\longrightarrow} \cdots \stackrel{\partial_1}{\longrightarrow} F_0 \longrightarrow \mathbb{Z}\longrightarrow 0\] 
	with $F_i= (\mathbb Z G)^{e_i}$ be the chain complex of the universal cover $\tilde X$. By the Hurewicz theorem $\pi_n(X) \simeq H_n(X)= \ker \partial _n$ and thus $X$ is aspherical if and only if $\partial_n$ is injective.
	
	Suppose $\ker \partial_n \not = \{0\}$ and consider $M= \ker \partial_{n-1} = \mathrm{im} \partial_n$. We apply Proposition \ref{free} to the $\mathbb Z G$- homomorphism $\partial_n: F_n \rightarrow M$, where $F_n=(\mathbb ZG)^{e_n}$ to deduce that $u:=d_H(M)< e_n[G:H]$ for some subgroup $H$.
	
	Choose a set of generators $\alpha_1, \ldots ,\alpha_u$ of the $\mathbb{Z}H$-module $M$. Let $Y$ be the cover of $X$ with degree $[Y:X]=[G:H]$ and $\pi_1(Y)=H$. Let $p: \tilde X \rightarrow Y$ be the universal covering map. Denote by $Y^{n-1}$ and $\tilde X^{n-1}$ the $(n-1)$-skeleta of $Y$ and $\tilde X$ respectively and observe that $\pi_{n-1}(Y^{n-1}) \simeq H_{n-1}(\tilde X^{n-1}) = \ker \partial_{n-1}=M$ by the Hurewicz theorem.
	Therefore for each $i=1, \ldots, u$ we can find a cellular map $j_i: S^{n-1} \rightarrow \tilde X^{n-1}$ representing $\alpha_i$.
	This means that $H_{n-1}(j_i)$ sends the generator of $H_{n-1}(S^{n-1})$ to the element $\alpha_i \in H_{n-1}(\tilde X^{n-1})=M$.
	
	We now attach $n$-dimensional cells $\sigma^n_i$ to $Y^{n-1}$ for $i=1,\ldots u$ with boundary attaching maps \[ S^{n-1} \stackrel{j_i}{\longrightarrow} \tilde X^{n-1} \stackrel{p}{\longrightarrow} Y^{n-1} \]	
	and define $Z:= Y^{n-1} \cup_{i=1}^u \sigma^n_i$. Note that since $Y^{n-1}= Z^{n-1}$ we have $\pi_i(Z)=\pi_i(Y)$ for $i=1, \ldots, n-2$. We claim that $\pi_{n-1}(Z)=\{0\}$. It is sufficient to prove that $H_{n-1}(\tilde Z)=\{0\}$ for the universal cover $\tilde Z$ of $Z$. Since the $(n-1)$-skeleta of $Z$ and $X$ coincide, the boundary maps $\partial_{n-1}$ on the chain complex of $\tilde Z$ and $\tilde X$ are the same and hence $\ker \partial_{n-1}=M$.
		On the other hand the boundary map $\partial'_n: (\mathbb ZH)^u \rightarrow M$ of degree $n$ of the chain complex of $\tilde Z$  is surjective since by construction its image contains the generators $\alpha_i$.
	Therefore $H_{n-1}(\tilde Z)=\{0\}$ and so $\tilde Z$ is $(n-1)$-connected as claimed.
	
	Note that $Z$ has $[G:H]e_i$ cells in dimension $i$ for $i=0,1 \ldots, n-1$ and $u$ cells in dimension $n$.
	Since $u< e_n[G:H]$ it follows that \[ \nu_n(H) \leq (-1)^n \chi(Z) = u + \sum_{i=0}^{n-1} (-1)^{n-i}e_i[G:H] < \nu_n(G)[G:H],\] contradiction. Therefore $H_n(\tilde X)=\{0\}$ and $X$ is a finite $K(G,1)$-complex of dimension $n$.
\end{proof}

\end{document}